\documentclass[leqno]{article}   
\usepackage[T1]{fontenc}
\usepackage[margin=1in,letterpaper]{geometry}
\usepackage[maxnames=99]{biblatex}
\usepackage{parskip}
\usepackage{mathtools}
\usepackage{amssymb}
\usepackage{amsthm}
\usepackage{thmtools}
\usepackage{tikz}
\usetikzlibrary{decorations.pathreplacing,graphs}
\usepackage{enumitem}
\usepackage{microtype}
\usepackage{graphicx}
\usepackage{xcolor}
\usepackage{transparent}
\usepackage[labelformat=simple]{subcaption}
     
\usepackage{bm}
\usepackage[colorlinks=true, allcolors=blue,unicode]{hyperref}
\usepackage[nameinlink]{cleveref}

\swapnumbers
\numberwithin{equation}{section}
\numberwithin{figure}{section}

\newtheorem{thm}[equation]{Theorem}
\crefname{thm}{Theorem}{Theorems}
\Crefname{thm}{Theorem}{Theorems}

\newtheorem*{thm*}{Theorem}

\newtheorem{lemma}[equation]{Lemma}
\crefname{lemma}{Lemma}{Lemmas}
\Crefname{lemma}{Lemma}{Lemmas}

\theoremstyle{definition}

\crefname{definition}{Definition}{Definitions}
\Crefname{definition}{Definition}{Definitions}

\newtheorem{ex}[equation]{Example}
\crefname{ex}{Example}{Examples}
\Crefname{ex}{Example}{Examples}

\newtheorem{rmk}[equation]{Remark}
\crefname{rmk}{Remark}{Remarks}
\Crefname{rmk}{Remark}{Remarks}

\newcommand{\Z}{\mathbf{Z}}

\newcommand{\R}{\mathbf{R}}
\newcommand{\C}{\mathbf{C}}
\newcommand{\perm}[1]{[#1]}
\newcommand{\symgrp}[1]{S_{#1}}
\newcommand{\rothe}{D}

\newcommand{\rank}[2]{\operatorname{rank}_{#1}#2}
\newcommand{\graph}[1]{G^{#1}}
\newcommand{\T}{{\mathsf{T}}}
\newcommand{\TxT}{{\mathsf{T}} \times {\mathsf{T}}}
\newcommand{\sm}{\setminus}
\newcommand{\dmax}{d_{\max}}
\newcommand{\edgecone}[1]{\sigma_{#1}^{\vee}}
\newcommand{\comp}[1]{\mathcal{C}(#1)} 
\newcommand{\msv}[1]{\overline{X_{#1}}} 
\newcommand{\Y}[1]{Y_{#1}}
\newcommand{\defn}[1]{{\bfseries{#1}}}
\newcommand{\M}{{\mathsf{M}}}
\newcommand{\edgeset}[1]{E(#1)}
\newcommand{\vertexset}[1]{V(#1)}
\newcommand{\edge}[2]{(#1 \to #2)}
\newcommand{\B}{{\mathsf{B}}}

\DeclarePairedDelimiter\abs{\lvert}{\rvert}

\DeclareMathOperator{\Spec}{Spec}
\DeclareMathOperator{\cone}{Cone}

\DeclareMathOperator{\ess}{Ess}

\DeclareMathOperator{\nw}{NW}
\DeclareMathOperator{\dom}{dom}
\DeclareMathOperator{\length}{\ell}

\DeclareMathOperator{\Inv}{Inv}
\DeclareMathOperator{\diag}{diag}

\title{Complexity of the zero set of a matrix Schubert ideal}
\author{Laura Escobar\thanks{\href{mailto:lauraescobar@ucsc.edu}{lauraescobar@ucsc.edu}. Mathematics Department, University of California, Santa Cruz} \and Cesar J.\ Meza\thanks{\href{mailto:c.j.meza@wustl.edu}{c.j.meza@wustl.edu}. Department of Mathematics, Washington University in St.~Louis}}
\date{}
\begin{document}
\maketitle

\begin{abstract}
    $T$-varieties are normal varieties equipped with an action of an algebraic torus $T$. When the action is effective, the complexity of a $T$-variety $X$ is 
    $\dim(X)-\dim(T)$. Matrix Schubert varieties, introduced by Fulton in 1992, are $T$-varieties consisting of $n \times n$ matrices satisfying certain constraints on the ranks of their submatrices. 
    In this paper, we focus on the complexity of certain torus-fixed affine subvarieties of matrix Schubert varieties. Concretely, given a matrix Schubert variety $\msv{w}$ where $w\in \symgrp{n}$, we study the complexity of $Y_w$ obtained by the decomposition $\msv{w} = Y_{w} \times \C^{k}$ with $k$ as large as possible. Building on results by Escobar--M\'{e}sz\'{a}ros and Donten-Bury--Escobar--Portakal, we show that for a fixed $n$, the complexity of $Y_{w}$ with respect to this action can be any integer between $0$ and $(n-1)(n-3)$, except $1$.
\end{abstract}

\section{Introduction}

A $T$-variety is a normal variety $V$ equipped with an action of an algebraic torus $T$, and the complexity of $V$ is the difference between the dimension of $V$ and the dimension of a maximal $T$-orbit. This nonnegative integer provides information on the combinatorial tools that can be applied to understand the variety \cite{Altmann:2006aa,Altmann:2012aa}. A general guiding principle is that the lower the complexity, the more amenable the variety is to combinatorial methods.
For example, $T$-varieties of complexity $0$ are precisely toric varieties, which are completely described using polyhedral objects such as polytopes and cones \cite{Cox_Little_Schenck}.

Flag varieties come equipped with the action of a torus, and it is natural to study the complexity of their torus-invariant subvarieties. In particular, there has been work classifying Schubert and Richardson varieties of a given complexity, e.g., \cite{can2023toricrichardsonvarieties,lee2024torusorbitclosuresflag}. 
A related problem is to carry out such a classification in the case of matrix Schubert varieties. These are affine varieties, introduced by Fulton \cite{Fulton:1992aa}, consisting of matrices that satisfy certain rank conditions. 
These varieties also come equipped with a torus action, and many interesting properties arise from this action, see, e.g., \cite{Knutson:2005aa}.

In this paper, we study the complexity of certain determinantal varieties closely related to matrix Schubert varieties.
Given a permutation $w \in \symgrp{n}$, the corresponding matrix Schubert variety $\msv{w}$ is a determinantal variety inside the space of complex-valued $n\times n$ matrices $\C^{n \times n}$.
This variety is isomorphic to the product of an affine variety $\Y{w}$ and the affine space $\C^{k}$ where $k$ is as large as possible. Studying the complexity of matrix Schubert varieties turns out not to be ideal, given that the factor of $\C^k$ makes low complexity difficult to achieve. Instead, we focus on $Y_w$ since its defining ideal coincides with that of $\msv{w}$.

Let $\T\simeq (\C^*)^n$ be the torus consisting of diagonal invertible $n\times n$ matrices.\footnote{In this paper, the tori we consider consist of invertible diagonal matrices. We will use the notation $\T$ when referring to such tori and reserve $T$ for an unspecified torus.}
The torus $\TxT$ acts on $\msv{w}$ via the map $((X,Y),M) \mapsto XMY^{-1}$. This action descends to an action on $\Y{w}$.
Characterizations have been given of those $Y_w$ that are toric, one using the Rothe diagram of $w$ by Escobar--Mészáros \cite[Theorem 3.5]{Escobar:2016aa} and another based on pattern avoidance by Stelzer \cite[Theorem 1.6]{Stelzer:2023aa}.
Moreover, in \cite[Theorem 3.14]{Donten-Bury:2023aa} it is shown that there are no $Y_w$ of complexity 1.
A natural question then is to study the set of nonnegative integers that can be achieved as the complexity of $Y_{w}$. On one hand, \cite[Theorem 3.15]{Donten-Bury:2023aa} proves that the set of nonnegative integers that can be achieved as the complexity of $Y_w$ only excludes $1$. 
However, in this result the permutations $w$ belong to some $\symgrp{n}$ where $n$ ranges over all positive integer multiples of $4$.
In this paper we are instead interested in determining the set of nonnegative integers that can be achieved as the complexity of $Y_w$ where $w$ belongs to a fixed $\symgrp{n}$ with $n \geq 4$.
Our main contribution is the following result:

\begin{thm*}[\Cref{thm: max-complexity}, \Cref{thm: all complexities}]
    Fix $n \geq 4$. With respect to the $\TxT$-action, the maximum over all $w \in \symgrp{n}$ of the complexity of the $T$-variety $Y_{w}$ is $(n-1)(n-3)$. 
    The unique permutation at which this maximum is achieved is $s_{n-1}=\perm{1,\dotsc,n-2,n,n-1}$. 
    In addition, for any $d \in \{0,2,3,\dotsc, (n-1)(n-3)\}$ there exists $w \in \symgrp{n}$ such that $Y_{w}$ has complexity $d$.
\end{thm*}
We remark that the varieties $Y_{1}$ and $Y_{w_0}$, where $1$ denotes the identity and $w_0$ the longest permutation, are points and so they are toric varieties, i.e., they have complexity $0$.

\section{Background}\label{section:background}

\subsection{\texorpdfstring{$T$-varieties}{T-varieties}}\label{section:T-varieties}

Let $T$ be an \defn{algebraic torus}. An affine normal variety $X$ is a \defn{$\boldsymbol{T}$-variety of complexity $\boldsymbol{d}$}, if it admits an effective $T$-action with $\dim(X) - \dim(T) = d$. Note that normal affine toric varieties are $T$-varieties of complexity $0$. In a sense, the complexity measures how far a $T$-variety is from being toric. For a more extensive exploration of $T$-varieties see \cite{Altmann:2006aa, Altmann:2012aa}.

Given a torus $T\simeq(\C^*)^n$, we can compute the complexity of a $T$-variety via a cone associated to $T$. Let $\M(T)\simeq\Z^n$ denote the \defn{character lattice} of $T$ and $\M(T)_{\R}\simeq\R^n$ the real vector space spanned by $\M(T)$. The \defn{weight cone} $\sigma$ of a torus action is the convex polyhedral cone generated by all weights of the action on $X$ in $\M(T)_{\R}$. For a general point $p \in X$, the closure of the torus orbit $\overline{T \cdot p}$ is the affine normal toric variety associated to the weight cone $\sigma$ and thus $\dim(\overline{T\cdot p}) = \dim(\sigma)$. 
When the action of $T$ on $\overline{T \cdot p}$ is effective, we have that $\dim(T) = \dim(\overline{T \cdot p})$. 
Therefore, the complexity of a $T$-variety $X$ is given by
\begin{equation}\label{eq: complexity polyhedral}
    d = \dim(X) - \dim(\sigma).
\end{equation}

If the action of $T$ is not effective, then the action of $T/S$, where $S$ is the point-wise stabilizer of $X$, is an effective action on $X$. 
Since the weight cone of $X$ with respect to this action is still $\sigma$, the complexity of the $T/S$-action is also given by \eqref{eq: complexity polyhedral}, see \cite[Section 2.1]{Donten-Bury:2023aa} for details.
For the remainder of this paper, whenever we have an ineffective $T$-action on $X$, we will abuse notation and refer to $X$ as a $T$-variety with complexity equal to that of the $T/S$-action.

\subsection{Matrix Schubert varieties}\label{subsec:MSV}
Now we focus our attention on a specific class of $T$-varieties called matrix Schubert varieties. 
In this subsection, we define our notation and conventions, provide background results, and introduce our torus action of interest.

Let $[n] \coloneq \{1, \dotsc, n\}$ and let $\symgrp{n}$ denote the symmetric group of permutations on $[n]$. For $w \in \symgrp{n}$ we write $w$ in \defn{one-line notation} as $ w = \perm{w(1),w(2),\dotsc, w(n)} = \perm{w_{1}, \dotsc, w_{n}}$. If $n < 10$, we will omit the brackets and commas and write $w = w_{1}w_{2} \dotsb w_{n}$. 
The set of inversions of a permutation $w \in \symgrp{n}$ is
\begin{equation*}
    \Inv(w) \coloneq \left\{(i,j) \in [n]^{2} : i < j, w(i) > w(j)\right\}.
\end{equation*}
    
The \defn{permutation matrix} of $w \in \symgrp{n}$, which by abuse of notation we also call $w$, is the $n\times n$ matrix with entries defined by
\begin{equation}
w_{ij}:= 
	\begin{cases}
		1,	&	\text{if } w(j) = i,\\
		0,	&	\text{otherwise.}
	\end{cases}    
\end{equation}
In other words, the permutation matrix associated to $w \in \symgrp{n}$ is the $n \times n$ matrix whose $j$th column is the $w_{j}$th standard basis vector for all $j \in [n]$. For example, the permutation matrix associated to $34512 \in \symgrp{5}$ is
\[
\begin{pmatrix}
    0 & 0 & 0 & 1 & 0\\
    0 & 0 & 0 & 0 & 1\\
    1 & 0 & 0 & 0 & 0\\
    0 & 1 & 0 & 0 & 0\\
    0 & 0 & 1 & 0 & 0
\end{pmatrix}.
\]

Let $\B_-$ denote the set of \defn{invertible lower triangular matrices} in $\C^{n \times n}$ and let $\B_+$ denote the set of \defn{invertible upper triangular matrices}. Consider the action of $\B_-\times\B_+$ on $\C^{n \times n}$ given by
\begin{equation}\label{eq:B-xB+-action}
    \begin{aligned}
        (\B_-\times\B_+) \times \C^{n\times n} &\to \C^{n\times n}\\
        ((X,Y),M)                              &\mapsto XMY^{-1}.
    \end{aligned}
\end{equation}
The orbit of a matrix $M$ under the $\B_-\times\B_+$-action is determined by certain rank conditions on submatrices of $M$. Permutation matrices form a set of representatives for the set of orbits that consist of nonsingular matrices.

To describe when a matrix $M$ is in the orbit of some permutation matrix $w \in \C^{n \times n}$, we first define submatrices $M_{\square}^{a,b}$ of $M$.
Given a matrix $M \in \C^{n \times n}$ and $a, b \in [n]$, let $M_{\square}^{a,b} \in \C^{a \times b}$ be the upper left submatrix of $M$ consisting of rows $1, \dotsc, a$ and columns $1, \dotsc, b$ as in \Cref{fig:submatrix}.
Let $\rank{M}{(a, b)}$ denote the \defn{rank} of $M_{\square}^{a, b}$. 
For a permutation matrix $w \in \C^{n \times n}$, a matrix $M$ is in the orbit $\B_{-}w\B_{+}$ if and only if $\rank{M}{(a, b)} = \rank{w}{(a, b)}$ for all $a, b \in [n]$.

\begin{figure}[ht]
\centering
    \begin{tikzpicture}[baseline=(O.base), scale = 1.2]
        \node(O) at (1,1) {};
	\draw (0,0) rectangle (2,2);
	\draw (0,2) rectangle (1.7, .8);
	\draw (-.2, .9) node {$a$};
	\draw (1.6, 2.2) node {$b$};
	\draw (.85, 1.4) node {$M_{\square}^{a, b}$};
    \end{tikzpicture}\\
\caption{The submatrix $M_{\square}^{a, b}$ of $M$. This figure is adapted from \cite{Escobar:2016aa}.}
\label{fig:submatrix}
\end{figure}
	
The \defn{matrix Schubert variety} associated to $w \in \symgrp{n}$ is 
the Zariski closure $\msv{w}\coloneq \overline{\B_-w\B_+}\subseteq \C^{n \times n}$. Fulton introduced matrix Schubert varieties in 1992 in his study of degeneracy loci of a map of flagged vector bundles \cite{Fulton:1992aa}. He described the ideals defining matrix Schubert varieties combinatorially using Rothe diagrams. 

The \defn{Rothe diagram} of $w\in \symgrp{n}$ is the set
\begin{equation}\label{eq:rothe}
    \rothe(w)\coloneq \{(i,j) \in [n]^2 : j < w^{-1}(i), w(j) > i\}.
\end{equation}
Equivalently, 
\begin{equation}\label{eq:rothe-inv}
    \rothe(w) = \{(w(j),i)\in [n]^2 : i < j, w(i) > w(j)\}
\end{equation}
which means entries of $\rothe(w)$ are in one-to-one correspondence with inversions of $w \in \symgrp{n}$. 
It follows that the Coxeter length $\length(w)$ of $w$ is equal to $\abs*{\rothe(w)}$.

We use an $n \times n$ grid to visualize the permutation matrix of $w \in \symgrp{n}$ and its associated Rothe diagram $\rothe(w)$. 
To illustrate $w$ using this grid, use matrix coordinates to place a $\bullet$ in position $(w(j), j)$ for each $j \in [n]$. 
In effect, this replaces each $1$ of the permutation matrix with a $\bullet$ and replaces the $0$s with empty boxes. 
Then from each $\bullet$ fire lasers south and east. The boxes not hit by a laser are elements in the Rothe diagram $\rothe(w)$. 
In other words, every element $(i,j)\in\rothe(w)$ has a $\bullet$ to its south and a $\bullet$ to its east. 
Note that each connected component of $\rothe(w)$ is a Young diagram in English notation.

\begin{figure}[ht]
\centering
\begin{tikzpicture}[baseline=(O.base),scale=1.3]
    \node(O) at (1,1) {};
    \fill[cyan!20]
        (0,2.5) rectangle (1,2)
        (0,2.5) rectangle (.5,1.5)
        (1.5,.5) rectangle (2,1);
    \draw[step=.5] (0,0) grid (2.5,2.5);
    \draw[draw=red]
        (.25,1.75) node {\scriptsize $(2,1)$}
        (.75,2.25) node {\scriptsize $(1,2)$}
        (1.75,.75) node {\scriptsize $(4,4)$}
        (.25,0) -- (.25,1.25) node {$\bullet$} -- (2.5,1.25)
        (.75,0) -- (.75,1.75) node {$\bullet$} -- (2.5,1.75)
        (1.25,0) -- (1.25,2.25) node {$\bullet$} -- (2.5,2.25)
        (1.75,0) -- (1.75,.25) node {$\bullet$} -- (2.5,.25)
        (2.25,0) -- (2.25,.75) node {$\bullet$} -- (2.5,.75);
\end{tikzpicture}
\caption{The Rothe diagram of the permutation $32154$.}
\label{fig:Rothe(32154)}
\end{figure}

The \defn{essential set} of a permutation $w \in \symgrp{n}$, denoted $\ess(w)$, is the set of all southeast corners of all connected components of $\rothe(w)$. 

\begin{ex}
    Let $w=32154\in\symgrp{5}$.
    The Rothe diagram of $w$ is $\rothe(w)=\{(1,1),(1,2),(2,1),(4,4)\}$, as illustrated by the set of blue boxes in \Cref{fig:Rothe(32154)}.
    The essential set of $w$ is $\ess(w)=\{(1,2),(2,1),(4,4)\}$.
\end{ex}

The essential set of $w$ can be used to define $\msv{w}$. The following theorem is written as in \cite{Portakal:2023aa}, but was originally stated and proved in \cite{Fulton:1992aa}.
\begin{thm}[\protect{\cite[Proposition 3.3, Lemma 3.10]{Fulton:1992aa}}] 
\label{thm:Fulton}
    The matrix Schubert variety $\msv{w}$ is an affine variety of dimension 
    $n^2-\abs*{\rothe(w)}$. It is defined as a scheme by the determinants
    encoding the inequalities ${\rank{M}{(a, b)} \leq \rank{w}{(a, b)}}$ for all 
    $(a, b) \in \ess(w)$.    
\end{thm}

We illustrate this theorem in the following example. 

\begin{ex}
    Consider the permutation $32154\in\symgrp{5}$.
    In \Cref{fig:Rothe(32154)}, we see that $\abs*{\rothe(32154)}=4$ and $\ess(32154)=\{(1,2),(2,1),(4,4)\}$.
    The matrix Schubert variety $\msv{32154}\subset\C^{5\times 5}$ is defined by the inequalities $\rank{M}{(1,2)}\leq \rank{32154}{(1,2)}=0$, $\rank{M}{(2,1)}\leq \rank{32154}{(2,1)}=0$, and $\rank{M}{(4,4)}\leq \rank{32154}{(4,4)}=3$.
    Thus, the defining ideal of $\msv{32154}$ is
    \begin{equation*}
        \left(z_{11},z_{12},z_{21},\det(M_{\square}^{4,4})\right)
        = \left(z_{11},z_{12},z_{21}, \det
        \begin{pmatrix}
            z_{11} & z_{12} & z_{13} & z_{14}\\
            z_{21} & z_{22} & z_{23} & z_{24}\\
            z_{31} & z_{32} & z_{33} & z_{34}\\
            z_{41} & z_{42} & z_{43} & z_{44}
        \end{pmatrix}
        \right) \subset \C[z_{11},\dotsc, z_{55}]
    \end{equation*}
    and $\dim(\msv{32154}) = 25 - \abs*{\rothe(32154)} = 21$.
\end{ex}

Lastly, we observe that matrix Schubert varieties are $T$-varieties. First note that, as explained in \cite{Knutson:2005aa} after Theorem 2.4.3, they are normal varieties.
Let $\T$ be the set of invertible $n\times n$ diagonal matrices.
We obtain a torus action on $\msv{w}$ by restricting the $\B_-\times\B_+$-action from \eqref{eq:B-xB+-action} to $\TxT$.

\section{The variety \texorpdfstring{$\Y{w}$}{Yw} and its complexity}
In this section we closely follow \cite{Escobar:2016aa,Donten-Bury:2023aa}.
However, the matrix Schubert varieties in \cite{Donten-Bury:2023aa} are defined using a $\B_+\times\B_+$-action on $\C^{n\times n}$.
This imposes southwest rank conditions instead of northwest rank conditions on submatrices.
We can translate from the conventions in this paper to the conventions set in \cite{Donten-Bury:2023aa} via the map $w \mapsto w_0w$, where $w_0$ is the longest permutation in $\symgrp{n}$.

Given $w \in \symgrp{n}$, let $\Y{w}$ be an affine variety such that $\msv{w} = \Y{w} \times \C^{k}$ where $k$ is as large as possible.
From this description and since $\msv{w}$ is normal, it follows that $\Y{w}$ is a normal variety.
Below we define $Y_w$ using diagrams constructed from the Rothe diagram $\rothe(w)$.
Once we describe the torus action in \Cref{sec:torus_action_Yw}, we will see that $\Y{w}$ is a $T$-variety.

If $(1,1)\in\rothe(w)$, then we call the connected component of $(1,1)\in\rothe(w)$ the \defn{dominant piece} $\dom(w)$ of $w$. 
If $(1,1)\notin\rothe(w)$, then we define $\dom(w)$ to be empty. 
Note that $(a, b) \in \dom(w)$ if and only if $\rank{w}{(a, b)} = 0$. 
The \defn{northwest diagram} of $w$, denoted $\nw(w)$, is the set of $(i, j)$ that are northwest of some element in $\ess(w)$. 
Finally, we define $L(w) \coloneq \nw(w) \sm \dom(w)$ and $L'(w) \coloneq \nw(w) \sm \rothe(w)$ to be the \defn{$\boldsymbol{L}$-diagram} and \defn{$\boldsymbol{L'}$-diagram} of $w$ respectively. 
\Cref{fig:32154-diagrams} illustrates the Rothe, northwest, $L$, and $L'$-diagrams of the permutation $32154$.
Since connected components of $\rothe(w)$ are Young diagrams in English notation, it follows that $\dom(w)$ and $\nw(w)$ are also Young diagrams. By construction, $L(w)$ is a skew diagram. 
However, $L'(w)$ is not necessarily a skew diagram, see for example, \Cref{fig:diagram-L'(15243)}.

\begin{rmk}\label{rmk:1jin}
    Note that from \eqref{eq:rothe} it is immediate that $\rothe(w)$ contains no elements of the form $(n,j)$ or $(i,n)$ with $i,j\in[n]$. The same claim follows for $L(w)$. We will use this observation later in the proof of \cref{thm: max-complexity}.
\end{rmk}

\begin{figure}[ht]
    \begin{subfigure}{.16\textwidth}
    \centering
    \scalebox{1}[-1]{\resizebox{\linewidth}{!}{
    \begin{tikzpicture}[baseline=(O.base), scale = 1.3]
    \node(O) at (1,1) {};
    \fill[cyan!20]
	   (0, 0) rectangle (.5, 1)
	   (0, 0) rectangle (1, .5)
          (1.5, 1.5) rectangle (2,2);
    \draw[step=.5] (0,0) grid (2.5,2.5);
    \end{tikzpicture}
    }
    }
    \caption{$\rothe(32154)$.}
    \label{fig:rothe(32154)}
    \end{subfigure}
    \hfill
    \begin{subfigure}{.16\textwidth}
    \centering
    \scalebox{1}[-1]{\resizebox{\linewidth}{!}{
    \begin{tikzpicture}[baseline=(O.base), scale=1.3]
    \node(O) at (1,1) {};
    \fill[orange!20]
        (0,0) rectangle (2,2);
    \draw[step=.5] (0,0) grid (2.5,2.5);
    \end{tikzpicture}
    }
    }
    \caption{$\nw(32154)$.}
    \label{fig:sw(32154)}
    \end{subfigure}
    \hfill
    \begin{subfigure}{.16\textwidth}
    \centering
    \scalebox{1}[-1]{\resizebox{\linewidth}{!}{
    \begin{tikzpicture}[baseline=(O.base), scale=1.3]
    \node(O) at (1,1) {};
    \fill[cyan!20]
        (0,1) rectangle (2,2)
        (.5,.5) rectangle (2,1)
        (1,0) rectangle (2,.5);
    \draw[step=.5] (0,0) grid (2.5,2.5);
    \end{tikzpicture}
    }
    }
    \caption{$L(32154)$.}
    \label{fig:L(32154)}
    \end{subfigure}
    \hfill
    \begin{subfigure}{.16\textwidth}
        \centering
        \scalebox{1}[-1]{\resizebox{\linewidth}{!}{
        \begin{tikzpicture}[baseline=(O.base), scale=1.3]
        \node(O) at (1,1) {};
        \fill[violet!20]
            (0,1) rectangle (1.5,2)
            (.5,.5) rectangle (2,1.5)
            (1,0) rectangle (2,.5);
        \draw[step=.5] (0,0) grid (2.5,2.5);
        \end{tikzpicture}
        }
        }
        \caption{$L'(32154)$.}
        \label{fig:L'(32154)}
    \end{subfigure}
    \hfill
    \begin{subfigure}{.16\textwidth}
        \centering
        \scalebox{1}[-1]{\resizebox{\linewidth}{!}{
        \begin{tikzpicture}[baseline=(O.base), scale=1.3]
        \node(O) at (1,1) {};
        \fill[violet!20]
            (0,0) rectangle (.5,2)
            (.5,0) rectangle (1,.5)
            (1,0) rectangle (1.5,1.5)
            (1.5,0) rectangle (2,1);
        \draw[step=.5] (0,0) grid (2.5,2.5);
        \end{tikzpicture}
        }
        }
        \caption{$L'(15243)$.}
        \label{fig:diagram-L'(15243)}
    \end{subfigure}
\caption{The Rothe diagram, northwest diagram, $L$-diagram, and $L'$-diagram of the permutation $32154\in\symgrp{5}$, and the $L'$-diagram of the permutation $15243\in\symgrp{5}$.}
\label{fig:32154-diagrams}
\end{figure}

Note that by \cref{thm:Fulton}, the determinantal ideal defining $\msv{w}$ depends only on the submatrices contained in $\nw(w)$. 
To construct $\Y{w}$, consider the image of $\msv{w}$ under the projection of $\C^{n^{2}}$ onto the linear subspace spanned by the elementary matrices whose entries are not in $\nw(w)$. 
Since these entries are free in $\msv{w}$, it follows that the projection is isomorphic to $\C^{n^2-\abs*{\nw(w)}}$. 
Then $\Y{w}$ is defined to be the projection onto the entries of $L(w)$. Therefore, it follows that $\msv{w} = \Y{w} \times \C^{n^2-\abs*{\nw(w)}}$ with
\begin{equation}\label{eq:dim(Y_w)}
    \dim(\Y{w}) = \underbrace{n^2 - \abs*{\rothe(w)}}_{\dim(\msv{w})} - (n^2 - \abs*{\nw(w)}) = \abs*{\nw(w)} - \abs*{\rothe(w)}
    = \abs*{L'(w)}.
\end{equation}

\begin{ex}
    Once again, consider the permutation $32154 \in \symgrp{5}$.
    \Cref{fig:32154-diagrams} illustrates that $\abs*{\rothe(32154)}=4$, $\abs*{\nw(32154)}=16$, and $\abs*{L'(32154)}=12$.
    In this case, $k=n^2 - \abs*{\nw(32154)} = 25-16=9$.
    The matrix Schubert variety associated to $32154$ can be written as $\msv{32154}=\Y{32154} \times \C^{9}$, where $\Y{32154}$ is defined by the ideal
    \begin{equation*}
        \left(\det
        \begin{pmatrix}
            0       & 0      & z_{13} & z_{14}\\
            0       & z_{22} & z_{23} & z_{24}\\
            z_{31}  & z_{32} & z_{33} & z_{34}\\
            z_{41}  & z_{42} & z_{43} & z_{44}
        \end{pmatrix}
        \right) \subset \C[z_{13},z_{14},z_{22},z_{23},z_{24},z_{31},z_{32},z_{33},z_{34},z_{41},z_{42},z_{43},z_{44}],
    \end{equation*}
    and $\dim(\Y{32154}) = \abs*{L'(32154)} = 12$.
\end{ex}

\subsection{The torus action on \texorpdfstring{$\Y{w}$}{Yw}}\label{sec:torus_action_Yw}

Let $w\in\symgrp{n}$.
Note that $\Y{w}$ is isomorphic to the subvariety of $\msv{w}$ obtained by setting 
$z_{ij} = 0$ for all $(i,j)\notin\nw(w)$. 
Thus, the $\B_-\times\B_+$-action on $\msv{w}$ described in \Cref{subsec:MSV} induces a $\B_-\times\B_+$-action on $\Y{w}$. The 
\defn{usual torus action} is the restriction to $\TxT$ of the $\B_-\times\B_+$-action on $\Y{w}$, where $\T$ is the set of {invertible $n\times n$ diagonal matrices}.
Concretely, given $M\in \C^{n \times n}$ and $(X,Y)\in \TxT$,
\begin{align*}
    (X, Y)\cdot M                     &\mapsto	XMY^{-1}.
\end{align*}
Throughout this paper, we consider the $T$-variety structure of $Y_w$ with respect to this torus action.
The torus $\TxT$ has character lattice $\M(\TxT)\cong \Z^{n} \times \Z^{n}$. 
Let $e_{1}, \dotsc, e_{n}, f_{1}, \dotsc, f_{n}$ denote the standard basis for
$\Z^{n} \times \Z^{n}$.
Let $X=\diag(s_1,\ldots,s_n)$ and $Y=\diag(t_1,\ldots,t_n)$. 
Since the $(i,j)$-coordinate of $XMY^{-1}$ is $s_{i}t_{j}^{-1}z_{ij}$, the \defn{weights} of the $\TxT$-action on 
$\C^{n \times n} = \Spec\left(\C[z_{11}, \dotsc, z_{nn}]\right)$ form the set
$\{e_{i} - f_{j} : i, j \in [n]\}$.
Since $\Y{w}$ can be obtained from $\msv{w}$ by setting 
$z_{ij} = 0$ for all $(i,j)\notin\nw(w)$, the \defn{weight cone} of the $\TxT$-action on $Y_{w}$ is 
\[
\sigma_{w} = \cone\left(\left\{e_{i} - f_{j} : (i, j) \in L(w)\right\}\right).
\]

It is useful to note that $\sigma_{w}$ is the edge cone constructed from a graph $G^w$, defined below. 
These graphs were first defined by Portakal in \cite{Portakal:2023aa}.
Concretely, let $G$ be a directed graph with vertex set $\vertexset{G}$ and edge set $\edgeset{G}$. The edge cone $\edgecone{G} \subseteq \M(T)_{\R}$ is given by
\[
\edgecone{G} = \cone\left(\left\{e_{i} - e_{j} : \edge{i}{j} \in 
E(G)\right\}\right),
\]
see, e.g., \cite{VV06}.
The following result gives a formula for the dimension of the edge cone.
\begin{lemma}[\protect{\cite[Lemma 2.3]{Donten-Bury:2023aa}}] \label{lem:dim(edgecone)}
Let $G$ be a directed acyclic graph with vertex set $\vertexset{G}$ and set $\comp{G}$ of connected components. The dimension of the edge cone $\edgecone{G} \subseteq \M(T)_{\R}$ is 
\[
\dim(\edgecone{G}) = \abs*{\vertexset{G}} - \abs*{\comp{G}}.
\]
\end{lemma}

Given $w \in \symgrp{n}$, let $\graph{w}$ be the directed acyclic bipartite graph with 
$\vertexset{\graph{w}} \subseteq \{1, \dotsc, n\} \sqcup \{\overline{1}, \dotsc, 
\overline{n}\}$ and 
$\edgeset{\graph{w}} = \{\, \edge{a}{\overline{b}} : (a, b) \in L(w) \,\}$
such that $\graph{w}$ has no isolated vertices.
Note that $\abs*{\vertexset{\graph{w}}}$ is equal to the number of nonempty rows plus the number of nonempty columns in $L(w)$.
By definition, $\sigma_{w} = \edgecone{\graph{w}}$.

\begin{figure}[ht]
\centering
\begin{tikzpicture}[]
\graph [simple, grow right=2cm, math nodes] {
    {1,2,3,4} ->[complete bipartite] {"\overline{1}", "\overline{2}", "\overline{3}", "\overline{4}"};
    2 -!- "\overline{1}";
    1 -!- {"\overline{1}", "\overline{2}"};
    };
\end{tikzpicture}
\caption{The bipartite graph $\graph{32154}$.}
\label{fig:graph(32154)}
\end{figure}

Following \cite[pg.~841]{Donten-Bury:2023aa}, $\Y{w}$ is a $T$-variety of complexity $d$ with respect to the $\TxT$ torus action if and only if
\begin{equation} \label{eq:dim(sigma)}
\dim(\sigma_{w}) = \dim(\Y{w}) - d = \abs{L'(w)} - d.
\end{equation}

Let $d_{w}$ denote the complexity of the $T$-variety $Y_{w}$ with respect to the $\TxT$ torus action. If the permutation is clear from context, we will drop the subscript and write $d$. Combining \eqref{eq:dim(sigma)} with \Cref{lem:dim(edgecone)}, we get that the complexity of $\Y{w}$ is given by
\begin{equation} \label{eq: complexity L'(w)}
    d_{w} = \abs{L'(w)} - \dim(\sigma_{w}) = \abs{L'(w)} - \abs*{\vertexset{\graph{w}}} + \abs*{\comp{\graph{w}}}.
\end{equation}

\begin{ex}
    Continuing with the permutation $32154 \in \symgrp{5}$, in \Cref{fig:L'(32154)} we see that $\abs*{L'(32154)} = 12$.
    Moreover, $\abs*{\vertexset{\graph{32154}}}=8$ and $\abs*{\comp{\graph{32154}}}=1$, see \Cref{fig:L(32154)} and \Cref{fig:graph(32154)}.
    Then, with respect to the $\TxT$-action, $\Y{32154}$ is a $T$-variety of complexity 
    \begin{equation*}
        d_{32154} = \abs*{L'(32154)} - 8 + 1 = 5.
    \end{equation*}
\end{ex}
Note that
\begin{equation*}
    L'(w) = L(w) \sm \rothe(w) = \left(\nw(w) \sm \dom(w)\right) \sm \rothe(w).
\end{equation*}
Since $\dom(w) \subseteq \rothe(w)$, it follows that $\abs*{L'(w)}=\abs*{L(w)} + \abs*{\dom(w)} - \abs*{\rothe(w)}$. Thus, we can write the complexity of a $\TxT$-variety $\Y{w}$ as
\begin{equation}\label{eq:complexity_L}
    d_w = \abs*{L(w)} + \abs*{\dom(w)} - \abs*{\rothe(w)} - \abs*{\vertexset{\graph{w}}} + \abs*{\comp{\graph{w}}}.
\end{equation}

\begin{figure}[ht]
    \begin{subfigure}{.22\textwidth}
    \centering
    \resizebox{\linewidth}{!}{
    \begin{tikzpicture}[baseline=(O.base), scale = 1.3]
    \node(O) at (1,1) {};
    \fill[cyan!20]
        (1, .5) rectangle (1.5, 1.5);
    \draw 
        (.25,2.25) node {$\bullet$}
        (.75,1.75) node {$\bullet$}
        (1.25,.25) node {$\bullet$}
        (1.75,1.25) node {$\bullet$}
        (2.25,.75) node {$\bullet$};
    \draw[step=.5] (0,0) grid (2.5,2.5);
    \end{tikzpicture}
    }
    \caption{$\rothe(12534)$}
    \end{subfigure}\hfill
    \begin{subfigure}{.22\textwidth}
    \centering
    \resizebox{\linewidth}{!}{
    \begin{tikzpicture}[baseline=(O.base),scale=1.3]
        \node(O) at (1,1) {};
        \fill[orange!20]
            (0,2.5) rectangle (1.5,.5);
        \draw[step=.5] (0,0) grid (2.5,2.5);
    \end{tikzpicture}
    }
    \caption{$\nw(12534)$}
    \end{subfigure}\hfill
    \begin{subfigure}{.22\textwidth}
    \centering
    \resizebox{\linewidth}{!}{
    \begin{tikzpicture}[baseline=(O.base),scale=1.3]
        \node(O) at (1,1) {};
        \fill[cyan!20]
            (0,2.5) rectangle (1.5,.5);
        \draw[step=.5] (0,0) grid (2.5,2.5);
    \end{tikzpicture}
    }
    \caption{$L(12534)$}
    \end{subfigure}\hfill
    \begin{subfigure}{.22\textwidth}
        \centering
        \resizebox{.75\linewidth}{!}{
        \begin{tikzpicture}[]
            \graph [simple, grow right=2cm, math nodes] {
            {1,2,3,4} ->[complete bipartite] {"\overline{1}", "\overline{2}", "\overline{3}"};
            };
        \end{tikzpicture}
        }
        \caption{$\graph{12534}$.}
        \label{fig:graph(12534)}
    \end{subfigure}
    \caption{The Rothe diagram, northwest diagram, $L$-diagram, and graph of the permutation $12534$.}
    \label{fig:12534}
\end{figure}

\begin{ex} \label{ex:54132-complexity}
    Consider the permutation $12534 \in \symgrp{5}$. 
    Using \Cref{fig:12534} and \eqref{eq:complexity_L} we see that the complexity of the $\TxT$-variety $\Y{12534}$ is
    \begin{align*}
        d_{12534}   &= \abs*{L(12534)} + \abs*{\dom(12534)} - \abs*{\rothe(12534)} - \abs*{\vertexset{\graph{12534}}} + \abs*{\comp{\graph{12534}}}\\
                    &= 12 + 0 - 2 - 7 + 1\\
                    &= 4.
    \end{align*}
\end{ex}

By definition, $\abs*{\edgeset{\graph{w}}}=\abs*{L(w)}$, so we can rewrite \eqref{eq:complexity_L} as
\begin{equation}\label{eq:complexity-cyclomatic}
    \begin{split}
    d_w&=\abs*{\edgeset{\graph{w}}} -\abs*{\vertexset{\graph{w}}} + \abs*{\comp{\graph{w}}} + \abs*{\dom(w)} - \abs*{\rothe(w)}\\
    &=\nu(\graph{w}) + \abs*{\dom(w)} - \abs*{\rothe(w)},
    \end{split}
\end{equation}
where $\nu(\graph{w}) = \abs*{\edgeset{\graph{w}}} -\abs*{\vertexset{\graph{w}}} + \abs*{\comp{\graph{w}}}$ is the cyclomatic number of the underlying undirected graph of $\graph{w}$.
The cyclomatic number $\nu(G)$ of a graph $G$ is the minimum number of edges needed to be removed from $G$ in order to make it acyclic.
The cyclomatic number $\nu(G)$ is also known as the first Betti number \cite{Whitney:1932aa}, nullity \cite{Whitney:1932aa}, or corank of $G$ \cite{Bollobas:1998aa}.

We will use the following lemma about $\nu(G)$ to prove \Cref{thm: max-complexity}.

\begin{lemma}\label{lem:cyclomatic number}
    The unique subgraph $G \subseteq K_{n-1,n-1}$ maximizing the cyclomatic number $\nu(G)$ is $K_{n-1,n-1}$.
\end{lemma}

\begin{proof}
    We will prove the statement algorithmically by taking an arbitrary nonempty subgraph of $K_{n-1,n-1}$ and performing operations on the graph that increase its cyclomatic number.
    Begin with a nonempty subgraph $G\subseteq K_{n-1,n-1}$.
    Next, add each $v \in K_{n-1,n-1}\sm G$ to $G$. 
    This operation preserves $\nu(G)$ since each new vertex increases $\abs*{\comp{G}}$ by $1$.
    For each isolated vertex in $G$, add an edge connecting it to another connected component of $G$.
    This again preserves the cyclomatic number $\nu(G)$ since each new edge reduces the number of connected components by $1$.
    Now we add edges to $G$ so that each connected component of $G$ is isomorphic to some complete bipartite graph $K_{a,b}$ with $a,b \geq 1$.
    This operation only increases $\nu(G)$ since the number of vertices and connected components remains constant.
    Finally, add all remaining edges between connected components $K_{a,b}$ so that $G$ becomes isomorphic to $K_{n-1,n-1}$.
    This increases $\nu(G)$ since reducing the number of connected components by $1$ corresponds to the addition of at least $2$ edges.
\end{proof}

\section{Main Results} \label{sec: Main Results}

In this section, we prove \Cref{thm: max-complexity} and \Cref{thm: all complexities}. Concretely, we determine the maximum complexity among all permutations in $\symgrp{n}$ and show that every integer value up to this maximum (excluding 1, see \cite[Theorem~3.14]{Donten-Bury:2023aa}) is the complexity of some $\TxT$-variety $\Y{w}$ with $w\in \symgrp{n}$.

We start by proving a theorem about 
\begin{equation*}
    \dmax(n)\coloneq \max\{d_{w} : w \in \symgrp{n}\}.
\end{equation*}

\begin{rmk}
    For every $w \in \symgrp{n}$ where $1\leq n\leq 3$, the associated $\TxT$-variety $Y_w$ has complexity $0$.
    For example, $Y_{132}$ is the toric variety defined by the ideal
    \begin{equation*}
        \left(\det
        \begin{pmatrix}
            z_{11} & z_{12}\\
            z_{21} & z_{22}
        \end{pmatrix}\right)
        =(z_{11}z_{22}-z_{12}z_{21})\subset \C[z_{11},z_{12},z_{21},z_{22}].
    \end{equation*}
    All other $Y_w$ where $w\in\symgrp{n}$ with $1\leq n\leq 3$ are points.    
    Thus, $\dmax(1)=\dmax(2)=\dmax(3)=0$.
\end{rmk}

\begin{thm} \label{thm: max-complexity}
    For $n \geq 4$, $\dmax(n)=(n - 1)(n - 3)$ and $s_{n-1} = \perm{1,\dotsc,n-2,n,n-1}$ is the unique permutation in $\symgrp{n}$ whose associated $\TxT$-variety achieves this complexity.
\end{thm}

\begin{proof}
    Note that 
    $\abs*{\dom(w)} - \abs*{\rothe(w)} \leq 0$ since $\dom(w) \subseteq \rothe(w)$. 
    Moreover, $\abs*{\dom(w)} - \abs*{\rothe(w)} = 0$ if and only if 
    $\abs*{L(w)} = 0$. When $\abs*{L(w)} = 0$ we have that $L(w) = \varnothing$ and $\Y{w}$ is the origin. Since the complexity of a point is $0$, it follows that if $w$ is such that 
    $\abs*{\dom(w)} - \abs*{\rothe(w)} = 0$, then the complexity of the $\TxT$-variety $\Y{w}$ is $d = 0$.  
    
    Let us now assume that $\abs*{\dom(w)} - \abs*{\rothe(w)} < 0$, $\abs*{\vertexset{\graph{w}}}\geq 2$, and $\abs*{\comp{\graph{w}}}\geq 1$.
    Using \eqref{eq:complexity-cyclomatic}, we have that the complexity $d$ of $\Y{w}$ is bounded by 
    \begin{equation}
        d \leq \nu(\graph{w}) - 1.
    \end{equation}

    By \Cref{lem:cyclomatic number}, we know that $\nu(\graph{w})$ is uniquely maximized when $\graph{w}\simeq K_{n-1,n-1}$.
    
    It remains to show that for $n\geq 4$, the unique $w\in\symgrp{n}$ such that $\graph{w}\simeq K_{n-1,n-1}$ and $\abs*{\dom(w)} - \abs*{\rothe(w)}=-1$ is $s_{n-1}$.
    Let $w \in \symgrp{n}$ such that $\graph{w}\simeq K_{n-1,n-1}$ and $\abs*{\dom(w)} - \abs*{\rothe(w)}=-1$.
    Since $\edge{1}{\overline{1}} \in \edgeset{\graph{w}}$, we know that $(1,1)\in L(w)$ and thus $\dom(w)=\varnothing$.
    It follows that $\abs*{\rothe(w)} = 1$.
    This means that $w$ has exactly one inversion.
    The only permutations $w \in \symgrp{n}$ with $\dom(w)=\varnothing$ and exactly one inversion are of the form $s_i$ where $1<i\leq n-1$.
    Among these permutations, only $s_{n-1}$ has an associated graph isomorphic to $K_{n-1,n-1}$.

    Observe that 
    \begin{align*}
        \nu(K_{n-1,n-1}) &= \abs*{\edgeset{K_{n-1,n-1}}} - \abs*{\vertexset{K_{n-1,n-1}}} + \abs*{\comp{K_{n-1,n-1}}}\\
        &= (n-1)^{2} - 2(n-1) + 1.
    \end{align*}
    Hence, $\dmax(n)=\nu(K_{n-1,n-1}) - 1 = (n-1)(n-3)$ and is uniquely achieved by the $\TxT$-variety associated to the permutation $s_{n-1}$.
\end{proof}

\begin{figure}[ht]
    \begin{subfigure}{.22\textwidth}
    \centering
    \resizebox{\linewidth}{!}{
    \begin{tikzpicture}[baseline=(O.base), scale = 1.3]
    \node(O) at (1,1) {};
    \fill[cyan!20]
        (1.5, .5) rectangle (2, 1);
    \draw[step=.5] (0,0) grid (2.5,2.5);
    \end{tikzpicture}
    }
    \caption{$\rothe(12354)$}
    \end{subfigure}\hfill
    \begin{subfigure}{.22\textwidth}
    \centering
    \resizebox{\linewidth}{!}{
    \begin{tikzpicture}[baseline=(O.base),scale=1.3]
        \node(O) at (1,1) {};
        \fill[orange!20]
            (0,2.5) rectangle (2,.5);
        \draw[step=.5] (0,0) grid (2.5,2.5);
    \end{tikzpicture}
    }
    \caption{$\nw(12354)$}
    \end{subfigure}\hfill
    \begin{subfigure}{.22\textwidth}
    \centering
    \resizebox{\linewidth}{!}{
    \begin{tikzpicture}[baseline=(O.base),scale=1.3]
        \node(O) at (1,1) {};
        \fill[cyan!20]
            (0,2.5) rectangle (2,.5);
        \draw[step=.5] (0,0) grid (2.5,2.5);
    \end{tikzpicture}
    }
    \caption{$L(12354)$}
    \end{subfigure}\hfill
    \begin{subfigure}{.22\textwidth}
        \centering
        \resizebox{.75\linewidth}{!}{
        \begin{tikzpicture}[]
            \graph [simple, grow right=2cm, math nodes] {
            {1,2,3,4} ->[complete bipartite] {"\overline{1}", "\overline{2}", "\overline{3}", "\overline{4}"};
            };
        \end{tikzpicture}
        }
        \caption{$\graph{12354}$.}
        \label{fig:graph(12354)}
    \end{subfigure}
    \caption{The Rothe diagram, northwest diagram, $L$-diagram, and graph of the permutation $12354$.}
    \label{fig:12354}
\end{figure}

\begin{ex}
    Consider the permutation $s_{n-1}\in\symgrp{5}$.
    \Cref{fig:12354} illustrates that $\graph{12354}\simeq K_{4,4}$ and $\abs*{\dom(12354)} - \abs*{\rothe(12354)}=-1$.
    The $\TxT$-variety associated to the permutation $s_{4}=12354$ achieves complexity $\dmax(5)=8$.
\end{ex}

Our last goal is to determine the integers that can appear as $d_w$ for some $w\in\symgrp{n}$, where $n$ is fixed. 
To do so, we will consider permutations whose Rothe diagram is a single box on the main diagonal and change them into permutations with smaller complexity. The following remark explains which permutations yield a Rothe diagram with this property and shows that the complexity associated with such permutations is equal to the maximum complexity for a smaller $n$.

\begin{rmk}\label{rmk: w0si}
    For $i \in [n-1]$, the permutation $s_i\in\symgrp{n}$ has Rothe diagram $\rothe(s_i)=\{(i,i)\}$
    In addition, for $i\geq 2$, the $\TxT$-variety $Y_{s_i}$ has complexity $d_{s_i} = i(i-2)=\dmax(i+1)$.
\end{rmk}

The following lemma describes how the complexity changes when boxes in a specific region of $[n]^2$ are added to the Rothe diagram of a permutation.
The proof of \cite[Theorem 3.15]{Donten-Bury:2023aa} discusses a similar problem. 
Concretely, given $u\in \symgrp{n}$ and $v\in \symgrp{k}$ it describes the complexity of the $\TxT$-variety $Y_{u\times w}$, where $u\times v\in \symgrp{n+k}$ is the image of $(u,v)$ under the embedding $\symgrp{n}\times \symgrp{k}\hookrightarrow \symgrp{n+k}$,
in terms of the complexities of $Y_u$ and $Y_v$.
Given a permutation $\alpha\in \symgrp{n}$, the following lemma instead deals with the situation of replacing a region in $[n]^2$ containing no elements of $\rothe(\alpha)$ with the diagram of a permutation $\beta\in \symgrp{m}$ with $m<n$.

\begin{lemma}\label{lem: antidiagonal diagrams}
    Let $\alpha$ be a permutation in $\symgrp{n}$ with associated $\TxT$-variety $\Y{\alpha}$ of complexity $d_{\alpha}$ such that $\rothe(\alpha)$ is nonempty and contained in the southeastern-most $k\times k$ submatrix. Let $m = n - k$ and let $\beta \in \symgrp{m}$. Then, the $\TxT$-variety $Y_w$ associated to the permutation 
    $w=\perm{\beta_1,\dotsc,\beta_m,\alpha_{m+1},\dotsc,\alpha_n}$ has complexity $d_\alpha -\abs*{\rothe(\beta)}$.
\end{lemma}

\begin{proof}
    First, note that if $\rothe(\alpha)$ is contained in the southeastern-most $k\times k$ submatrix, then $\alpha_i=i$ for all $i \leq m$.
    Then by construction, the Rothe diagram of $w$ is as in \Cref{fig: alpha beta diagram} where the area labeled $\varnothing$ has no boxes. 

\begin{figure}[ht]
    \centering
    \usetikzlibrary {decorations.pathmorphing}
    \begin{tikzpicture}[scale = 0.08cm]
    \draw[step=0.5cm,color=black] (0,0) grid (1,1);
    \node at (+0.25,+0.75) {$\rothe(\beta)$};
    \node at (.25,-.1) {$m$};
    \node at (1.1,.25) {$k$};
    \node at (+0.75,+0.75) {$\varnothing$};
    \draw[decoration={brace,raise=1pt},decorate]
        (0,0) -- node[left=1pt] {$n$} (0,1);
    \node at (.75, -.1) {$k$};
    \node at (1.1,.75) {$m$};
    \node at (+0.25,+0.25) {$\varnothing$};
    \node at (+0.75,+0.25) {$\rothe(\alpha)$};
    \end{tikzpicture}
    \caption{The Rothe diagram of $w = [\beta_{1}, \dotsc, \beta_{m}, \alpha_{m + 1}, \dotsc, \alpha_{n}]$.}
    \label{fig: alpha beta diagram}
\end{figure}

Since the boxes in $\rothe(\alpha)$ are southeast of all boxes in the area labeled $\rothe(\beta)$, we know that $\nw(w)=\nw(\alpha)$.
Moreover, because $\rothe(w)$ is the union of $\rothe(\alpha)$ and $\rothe(\beta)$, we know that $\abs*{\rothe(w)} = \abs*{\rothe(\alpha)} + \abs*{\rothe(\beta)}$.
 
Note that $L(w) = L(\alpha) \sm \dom(\beta)$. 
Since $\rothe(\alpha)\neq\varnothing$ and $\dom(\beta)$ is contained in the northwestern-most $m\times m$ submatrix, we have that $L(w)$ contains the set $\{(1,m+1),\dotsc,(m+1,m+1),\dotsc,(m+1,1)\}$.
It follows that the number of nonempty rows and nonempty columns in $L(w)$ equals that of $L(\alpha)$. Therefore, $\abs*{\vertexset{\graph{w}}} = \abs*{\vertexset{\graph{\alpha}}}$. 

Since $\dom(\alpha) = \varnothing$ and $\rothe(\alpha) \neq \varnothing$ we know that $\graph{\alpha}$ has one connected component. 
In addition, since $\{(1,m+1),\dotsc,(m+1,m+1),\dotsc,(m+1,1)\}\in L(w)$, we know that $\graph{w}$ also has one connected component.
Using \eqref{eq: complexity L'(w)} in combination with the fact that $L'(w)=\nw(w)\sm\rothe(w)$, we have that the complexity of $\Y{w}$ is given by
\begin{equation}
    \begin{aligned}
        d_{w} &= \abs*{\nw(w)} - \abs*{\rothe(w)} - \abs*{\vertexset{\graph{w}}} 
        + \abs*{\comp{\graph{w}}}\\
      &= \abs*{\nw(\alpha)} - (\abs*{\rothe(\alpha)} + \abs*{\rothe(\beta)}) 
        - \abs*{\vertexset{\graph{\alpha}}} + \abs*{\comp{\graph{\alpha}}}\\
      &= d_{\alpha} - \abs*{\rothe(\beta)}.
    \end{aligned}
\end{equation}
\end{proof}

We are now ready to complete the proof of our main result.

\begin{thm} \label{thm: all complexities}
    Fix $n \geq 4$. For any $d \in \{0, 2, 3, \dotsc, (n-1)(n-3)\}$ there exists $w\in \symgrp{n}$ such that $\Y{w}$ has complexity $d$.
\end{thm}

\begin{proof}
     From \Cref{thm: max-complexity} we know that the $\TxT$-variety associated to $s_{n-1}\in\symgrp{n}$ has complexity $\dmax(n)=(n-1)(n-3)$.
     Moreover, $\rothe(s_{n-1}) = \{(n-1,n-1)\}$ is contained in the southeasternmost $2\times 2$ submatrix.
     By \Cref{lem: antidiagonal diagrams}, for any permutation $\beta \in\symgrp{n-2}$ there exists a permutation $w\in\symgrp{n}$ whose associated $\TxT$-variety has complexity $\dmax(n)-\abs*{\rothe(\beta)}$.
     Since $\abs*{\rothe(\beta)}$ is the number of inversions of $\beta$, we know that $\abs*{\rothe(\beta)}\leq \frac{(n-2)(n-3)}{2}$.
     Therefore, we can achieve any complexity between $\dmax(n)$ and $\dmax(n)-\frac{(n-2)(n-3)}{2}$.

     Define $f(k)$ to be the difference between $\dmax(k)$ and the maximum number of inversions of a permutation $\beta\in\symgrp{k-2}$.
     Namely,
     \begin{equation}
         f(k)=\dmax(k)-\frac{(k-2)(k-3)}{2}=\frac{k(k-3)}{2},
     \end{equation}
     for $3\leq k\leq n$.
     Recall from \Cref{rmk: w0si} that the $\TxT$-variety associated to the permutation $s_{k-1}\in\symgrp{n}$ has complexity $\dmax(k)$.
     Therefore, every integer in $\bigcup_{k=3}^{n}[f(k),\dmax(k)]$ is the complexity of some $\TxT$-variety $Y_w$ with $w\in \symgrp{n}$.
     In \Cref{fig:intervals}, we illustrate the intervals $[f(k),\dmax(k)]$ for $3\leq k\leq 8$.
     
     \begin{figure}[h]
    \centering
    \begin{tikzpicture}[baseline=(O.base), scale = .35]
        \draw (-1,0)-- (36,0);
        \foreach \x in {0,1,2,3,4,5,8,9,14,15,20,24,35} {
            \draw (\x,0.25) -- (\x,-0.25) node[below] {$\x$};
            }
        \draw[fill] (0,1) circle (.125) node[right] {$f(3)=\dmax(3)$};
        \draw[fill=white] (1,2.5) circle (.125) node[right, font={\footnotesize}] {nonexistent};
        \draw[thick] [Bar-Bar] (2,4) node[left] {$f(4)$} -- (3,4) node[right] {$\dmax(4)$};
        \draw[fill] (4,5.5) circle (.125) node[right]  {$d_{12534678}$};
        \draw[thick] [Bar-Bar] (5,7) node[left] {$f(5)$} -- (8,7) node[right] {$\dmax(5)$};
        \draw[thick] [Bar-Bar] (9,8.5) node[left] {$f(6)$} -- (15,8.5) node[right] {$\dmax(6)$};
        \draw[thick] [Bar-Bar] (14,10) node[left] {$f(7)$} -- (24,10) node[right] {$\dmax(7)$};
        \draw[thick] [Bar-Bar] (20,11.5) node[left] {$f(8)$} -- (35,11.5) node[right] {$\dmax(8)$};
    \end{tikzpicture}
    \caption{The set of intervals $[f(k), \dmax(k)]$ for $3\leq k \leq 8$ and the permutation $w\in\symgrp{8}$ associated to a complexity-$4$ $\TxT$-variety.}
    \label{fig:intervals}
\end{figure}

     Note that $\dmax(3)=0$ and $\dmax(n)=(n-1)(n-3)$. 
     Thus if $\dmax(k-1)\geq f(k)-1$ for every $4 \leq k \leq n$, then we are done.
     In fact, this inequality holds for $6 \leq k \leq n$.
     It follows that $\dmax(k-1) < f(k)-1$ at $k=4$ and $k=5$.

     For $k=4$, we have that $\dmax(3)=0$ and $f(4)=2$.
     In principle, we could be missing a $\TxT$-variety of complexity $1$.
     However, \cite[Theorem~3.14]{Donten-Bury:2023aa} shows that there are no such varieties.

     For $k=5$, we have that $\dmax(4)=3$ and $f(5)=5$.
     Thus, the argument outlined above does not capture a $\TxT$-variety of complexity $4$.
     However, in \Cref{ex:54132-complexity}, we verified that $Y_{12534}\subset \C^{5\times 5}$ is a complexity-$4$ $\TxT$-variety.
     Therefore, for $n\geq 6$, the $\TxT$-variety associated to the image of $12534$ under the standard embedding of $\symgrp{5}$ into $\symgrp{n}$ has complexity $4$.
     See \Cref{fig:12534678-diagrams} for the case when $n=8$.
\end{proof}

\begin{figure}[ht]
    \begin{subfigure}{.22\textwidth}
    \centering
    \scalebox{1}[-1]{
    \resizebox{\linewidth}{!}{
    \begin{tikzpicture}[baseline=(O.base), scale = 1.3]
    \node(O) at (1,1) {};
    \fill[cyan!20]
	   (1, 1) rectangle (1.5, 2);
    \node at (.25,.25) {$\bullet$};
    \node at (.75,.75) {$\bullet$};
    \node at (1.25,2.25) {$\bullet$};
    \node at (1.75,1.25) {$\bullet$};
    \node at (2.25,1.75) {$\bullet$};
    \node at (2.75,2.75) {$\bullet$};
    \node at (3.25,3.25) {$\bullet$};
    \node at (3.75,3.75) {$\bullet$};
    \draw[step=.5] (0,0) grid (4,4);
    \end{tikzpicture}
    }
    }
    \caption{$\rothe(12534678)$.}
    \label{fig:rothe(12534678)}
    \end{subfigure}
    \hfill
    \begin{subfigure}{.22\textwidth}
    \centering
    \scalebox{1}[-1]{
    \resizebox{\linewidth}{!}{
    \begin{tikzpicture}[baseline=(O.base), scale=1.3]
    \node(O) at (1,1) {};
    \fill[orange!20]
        (0,0) rectangle (1.5,2);
    \draw[step=.5] (0,0) grid (4,4);
    \end{tikzpicture}
    }
    }
    \caption{$\nw(12534678)$.}
    \label{fig:nw(12534678)}
    \end{subfigure}
    \hfill
    \begin{subfigure}{.22\textwidth}
    \centering
    \scalebox{1}[-1]{
    \resizebox{\linewidth}{!}{
    \begin{tikzpicture}[baseline=(O.base), scale=1.3]
    \node(O) at (1,1) {};
    \fill[cyan!20]
        (0,0) rectangle (1.5,2);
    \draw[step=.5] (0,0) grid (4,4);
    \end{tikzpicture}
    }
    }
    \caption{$L(12534678)$.}
    \label{fig:L(12534678)}
    \end{subfigure}
    \hfill
    \begin{subfigure}{.22\textwidth}
        \centering
        \resizebox{.75\linewidth}{!}{
        \begin{tikzpicture}[]
            \graph [simple, grow right=2cm, math nodes] {
            {1,2,3,4} ->[complete bipartite] {"\overline{1}", "\overline{2}", "\overline{3}"};
            };
        \end{tikzpicture}
        }
        \caption{$\graph{12534678}$.}
        \label{fig:graph(12534678)}
    \end{subfigure}
\caption{The Rothe diagram, northwest diagram, $L$-diagram, and graph of the permutation $12534678\in\symgrp{8}$.}
\label{fig:12534678-diagrams}
\end{figure}

\section*{Acknowledgments}
LE and CM were partially funded by NSF CAREER grant DMS-2142656, DMS-2521270. We thank Akihiro Higashitani and John Shareshian for useful conversations. We also thank Sanah Suri for helping review code to compute examples.
We are very grateful to the anonymous referees for their constructive comments, all of which improved the manuscript.
\printbibliography
\end{document}